\documentclass[12pt,leqno]{amsart}
\usepackage[all]{xy}
\usepackage{amssymb,amsmath,amsthm,mathrsfs}
\usepackage{cases}
\usepackage{enumerate}
\usepackage{hyperref}

\numberwithin{equation}{section}
\allowdisplaybreaks[3]
\makeatother
\newtheorem{thm}{Theorem}[section]

\newtheorem{lem}[thm]{Lemma}

\theoremstyle{definition}

\numberwithin{equation}{section}

\theoremstyle{definition}
\newtheorem*{ack}{Acknowledgements}

\newcommand{\lab}{\label}

\newcommand{\eps}{\varepsilon}

\newcommand{\si}{\sigma}
\newcommand{\Gm}{\Gamma}

\newcommand{\abs}[1]{\left\lvert#1\right\rvert}
\newcommand{\re}{\textup{Re}}
\newcommand{\im}{\textup{Im}}

\begin{document}

\title[Bounds for double $L$-functions]{Bounds for double $L$-functions}

\author[Y. Toma]{Yuichiro Toma}
\address{Graduate School of Mathematics, Nagoya University, Chikusa-ku, Nagoya 464-8602, Japan.}

\email{m20034y@math.nagoya-u.ac.jp}  

\makeatletter
\@namedef{subjclassname@2020}{\textup{2020} Mathematics Subject Classification}
\makeatother

\begin{abstract} Double $L$-functions are the generalization of Dirichlet $L$-functions to two variable functions. We investigate the order estimation of double $L$-functions, and give upper bounds which are explicit in conductor aspect.
\end{abstract}

\maketitle
\section{Introduction}
For complex variables $s_1=\sigma_1+it_1, s_2=\sigma_2+it_2$, the double zeta-function is defined as
\[
\zeta_2(s_1,s_2):=\sum_{m=1}^\infty \sum_{n=1}^\infty \frac{1}{m^{s_1}(m+n)^{s_2}}
\]
for $\sigma_2>1,\sigma_1+\sigma_2>2$. This is a generalization of the Riemann zeta-function $\zeta(s) :=\sum_{n\leq 1}n^{-s}$ to a two variable function. In 1949, Atkinson \cite{A49} divided $\zeta(s_1)\zeta(s_2)$ into 
\[
\zeta(s_1)\zeta(s_2) = \zeta_2(s_1,s_2)+\zeta_2(s_2,s_1)+\zeta(s_1+s_2).
\]
Nowadays, this is called Atkinson's dissection. So it implies that the double zeta-function can be regarded as the non-diagonal term of the product of two Riemann zeta-functions. He studied $\zeta_2(s_1,s_2)$ in the research on the mean value formulas for the Riemann zeta-function. 

In this direction, the growth order of the double zeta-function has been studied. In \cite{IM03}, Ishikawa and Matsumoto first gave upper bound estimates of $\zeta_2(it,i\alpha t)$ for $\alpha \neq \pm1$. After that,  Kiuchi and Tanigawa \cite{KT06} gave upper bounds of $\zeta_2(s_1,s_2)$ for $0 \leq \sigma_1\leq 1, 0 \leq \sigma_2 \leq1$ by using double exponential sums due to Kr\"{a}tzel \cite{K88}. For example, they showed that 
\[
\zeta_2(1/2+it_1,1/2+it_2) \ll \abs{t_1}^\frac{1}{3} \log^2 \abs{t_1}.
\]
In 2011, Kiuchi, Tanigawa and Zhai \cite{KTZ11} proved that the exponent $1/3$ is best possible under a certain condition. They showed that 
\[
\zeta_2(1/2+it_1,1/2+it_2) = \Omega(\abs{t_1}^\frac{1}{3}\log\log \abs{t_1} )
\]
holds for $\abs{t_2} \ll\abs{t_1}^{\frac{1}{6}-\varepsilon}$. However, if $t_1,t_2$ do not satisfy the condition, then the exponent can be reduced. Indeed, Banerjee, Minamide and Tanigawa \cite{BMT20} refined the method of \cite{KTZ11} to obtain that
\[
\zeta_2(1/2+it_1,1/2+it_2) \ll \abs{t_1}^{\mu(\frac{1}{2})+\eps} \abs{t_2}^{\mu(\frac{1}{2})+\eps}
\]
for $\abs{t_2}^{\frac{1/2-\mu(1/2)}{1+\mu(1/2)}} \ll \abs{t_1} \ll\abs{t_2}^{\frac{1+\mu(1/2)}{1/2-\mu(1/2)}}$, where $\mu(\si) := \inf \{ c\geq 0 \mid \zeta(\si+it) \ll \abs{t}^c\}$. It is known that $\mu(\frac{1}{2})\leq \frac{13}{84}$ due to Bourgain \cite{Bo17}. Ultimately, if the Lindel\"of hypothesis is true, we can take $\mu(\frac{1}{2}) =0$.

In this paper, we investigate the upper bound estimates of double $L$-functions. Let $q$ be a positive integer and $\chi_1, \chi_2$ be Dirichlet characters with the same modulus $q \geq 2$. Then the double $L$-function associated with $\chi_1,\chi_2$ is defined as 
\begin{align}
\lab{DLF}
L_2 (s_1,s_2;\chi_1,\chi_2) &:=\sum_{m=1}^\infty \sum_{n=1}^\infty \frac{\chi_1(m) \chi_2(n)}{m^{s_1}(m+n)^{s_2}}
\end{align}
for $\sigma_2>1,\sigma_1+\sigma_2>2$.  
Arakawa and Kaneko \cite{AK04} and several authors studied its special values at positive integers which are called multiple $L$-values. 

On the other hand, from the analytic viewpoint, it is known that if $\chi_1,\chi_2$ are both primitive, then (\ref{DLF}) is entire on $\mathbb{C}^2$ space, due to Mastumoto and Tanigawa \cite{MT01}. Moreover, they studied the upper bound of (\ref{DLF}), and obtained that if $\chi_2$ is nontrivial, then for a fixed small positive number $\eta$
\begin{align}
\label{bound of MT}
L_2 (s_1,s_2;\chi_1,\chi_2) 
&\ll (1+|t_2|)^{\theta_2(\eta)+\frac{1}{2}+\max \{0, \frac{1}{2}-\sigma_2+\eta\} } 
\end{align}
holds in the region $\mathcal{D} =\{(s_1, s_2) \mid \sigma_1+\sigma_2>1+\eta, \sigma_2>\eta \}$, where 
\[
\theta_2(\si) := \inf \{\alpha \geq 0 \mid L(\si+it,\chi_2) \ll \abs{t}^{\alpha} \}.
\]
It should be remarked that they neglected the conductor aspect because they gave the upper bounds of not only (\ref{DLF}) but also more general double sums satisfying some fundamental axioms. In this paper, we show the following.

\begin{thm}
\lab{thm:1}
Let $(s_1,s_2) \in \mathbb{C}^2$ with $\sigma_1+\sigma_2>0, 0<\sigma_2<1$ and $\abs{t_2} \geq 2$. Suppose that $\chi_1, \chi_2$ are primitive Dirichlet characters modulo $q$. Then we have
\begin{align*}
L_2 (s_1,s_2;\chi_1,\chi_2) &\ll (q\abs{t_2})^{\frac{1}{2}+\delta+\eps} & if &\quad \chi_2(-1)=1, \\
L_2 (s_1,s_2;\chi_1,\chi_2) &\ll (q\abs{t_2})^{\frac{1}{2}+\delta+\eps} +(1+\abs{t_1+t_2})q^{\frac{3}{2}+\eps} (q\abs{t_2})^{-\min \{1,\frac{\si_1+\si_2}{2}\}} & if &\quad \chi_2(-1)=-1,
\end{align*}
where $\delta= \max \{ 0,1-\si_1-\si_2 \}$ and the implicit constants are independent of $q, t_1$ and $t_2$.
\end{thm}
When $\chi_1,\chi_2$ are both primitive, comparing Theorem \ref{thm:1} with (\ref{bound of MT}), we succeed in extending the region where the order estimates hold to $\si_1+\si_2>0$ if $0<\si_2<1$. Moreover, when $0<\si_2<1$, our bounds are better than (\ref{bound of MT}) in $\mathcal{D}$, even including the conductor aspect when $q \ll \abs{t_2}^{2\theta_2(\eta)-\eps}$ if $\chi_2(-1)=1$, and when $q \ll \abs{t_2}^{\theta_2(\eta)-\eps}/\abs{t_1+t_2}$ if $\chi_2(-1)=-1$.

In order to obtain the above upper bounds, we prove more precisely the following.
\begin{thm}
\lab{thm:2}
Let $(s_1,s_2) \in \mathbb{C}^2$ with $\sigma_1+\sigma_2>0$ and $0<\sigma_2<1$. Suppose that $\chi_1, \chi_2$ are primitive Dirichlet characters modulo $q$. Then 
\begin{align*}
L_2 (s_1,s_2;\chi_1,\chi_2) 
&=\frac{2(2\pi)^{s_2-1}\tau(\chi_2)\Gm(1-s_2)}{q^{s_2}}\sum_{mn\leq \frac{q|t_2|}{2\pi}} \frac{\chi_1(m) \overline{\chi_2}(n) \sin(\frac{\pi s_2}{2}+\frac{2\pi mn}{q})}{m^{s_1}n^{1-s_2}} \\
&\quad+O\left(q^{\frac{1}{2}} (q|t_2|)^{\delta+\eps}\right)
\end{align*}
holds if $\chi_2(-1)=1$, and 
\begin{align*}
L_2 (s_1,s_2;\chi_1,\chi_2) 
&=\frac{2(2\pi)^{s_2-1}\tau(\chi_2)\Gm(1-s_2)}{iq^{s_2}}\sum_{mn\leq \frac{q|t_2|}{2\pi}} \frac{\chi_1(m) \overline{\chi_2}(n) \cos(\frac{\pi s_2}{2}+\frac{2\pi mn}{q})}{m^{s_1}n^{1-s_2}} \\
&\quad+O\left((1+\abs{t_1+t_2})q^{\frac{3}{2}+\varepsilon}(q|t_2|)^{-\min \{1,\frac{\sigma_1+\sigma_2}{2} \}+\varepsilon }\right)+O\left(q^{\frac{1}{2}} (q|t_2|)^{\delta+\eps}\right)
\end{align*}
holds if $\chi_1(-1)=-1$, where $\tau(\chi_2)$ is the Gauss sum, $\delta =\max \{0,1-\sigma_1-\sigma_2 \} $ and implicit constants are
independent of $q, t_1$ and $t_2$.
\end{thm}

Since $2(2\pi)^{s-1}\Gm(1-s)\sin (\pi s/2), 2(2\pi)^{s-1}\Gm(1-s)\cos (\pi s/2)
\asymp t^{\frac{1}{2}-\sigma}$, we can find that 
\begin{align*}
\lab{upper bound}
&\frac{2(2\pi)^{s_2-1}\tau(\chi_2)\Gm(1-s_2)}{q^{s_2}} \sum_{mn\leq \frac{q|t_2|}{2\pi}} \frac{\chi_1(m) \overline{\chi_2}(n)}{m^{s_1}n^{1-s_2}} \left\{
\begin{array}{ll}
\sin\frac{\pi s_2}{2} \\
\cos\frac{\pi s_2}{2} 
\end{array}
\right.
\ll (q|t_2|)^{\frac{1}{2}+\delta+\eps}.
\end{align*}
Hence Theorem \ref{thm:2} implies Theorem \ref{thm:1}. Also, Theorem \ref{thm:2} implies that the asymptotic behaviour of double $L$-functions changes according to the parity of $\chi_2$.

Theorem \ref{thm:2} generalizes Theorem 2 of \cite{KTZ11} if $\chi_2$ is even. In the case, we can prove Theorem \ref{thm:2} by the same argument as in Theorem 2 of \cite{KTZ11}. However, in the case that $\chi_2$ is odd, an extra term remains. So we need Lemma \ref{approximation formula} to evaluate the extra term (see (\ref{J_5 chi_2:odd})). 

Moreover, in \cite{KTZ11}, they used an expression of the double zeta-function in terms of Kummer's confluent hypergeometric function ${}_1F_1(a,c;x)
$ (see \cite[Remark 1]{KTZ11}). On the other hand, in the present paper we use different confluent hypergeometric function which is defined as
\begin{equation}
\label{confluent hypergeometric function}
\Psi(a,c;x) := \frac{\Gm(1-c)}{\Gm(a-c+1)} {}_1F_1(a,c;x)+\frac{\Gm(c-1)}{\Gm(a)}x^{1-c} {}_1F_1(a-c+1,2-c:x).
\end{equation}
In order to obtain an expression of (\ref{DLF}) in terms of (\ref{confluent hypergeometric function}), we use the following interpolated double $L$-function as
\begin{align}
\lab{MT-MLF}
\widetilde{L}_{2,z} (s_1,s_2;\chi_1,\chi_2) &=\sum_{m,n=1}^\infty \frac{\chi_1(m) \chi_2(n)}{m^{s_1} n^z (m+n)^{s_2}}
\end{align}
for $z \in \mathbb{C}$, instead of (\ref{DLF}). This double $L$-series was introduced by Wu in his unpublished master thesis \cite{Wu03}, and it is called the Mordell-Tornheim double $L$-function. Wu showed that if neither characters $\chi_1$ nor $\chi_2$ are principal, then $\widetilde{L}_{2,z} (s_1,s_2;\chi_1,\chi_2)$ is entire. Also, it is easy to find that 
\[
\widetilde{L}_{2,0}(s_1,s_2;\chi_1,\chi_2)=L_2 (s_1,s_2;\chi_1,\chi_2).
\]
By using this relation we obtain an integral expression of (\ref{DLF}) in Lemma \ref{lem:L and conf hypergeometric-2}.

\section{Preliminaries}
In this section, we prove some lemmas. The first two lemmas are necessary to prove Theorem \ref{thm:2} when $\chi_2(-1)=-1$.
\begin{lem}
\label{hyperbola character sum}
Let $\chi_1,\chi_2$ be non-principal Dirichlet characters modulo $q>1$. Let $\tau, \xi$ be a large parameter such that $\tau \geq \xi$. Then we have
\begin{align*}
\sum_{\substack{mn \leq \tau \\ m \leq \xi}} \chi_1(m)\chi_2(n) &\ll \xi q^\frac{1}{2}\log q. 
\end{align*}
\end{lem}
\begin{proof}
By arranging terms and using the P\'{o}lya--Vinogradov inequality, we have
\[
\sum_{m \leq \xi} \chi_1(m) \sum_{n \leq \frac{\tau}{m}} \chi_2(n) \ll \sum_{m \leq \xi} q^\frac{1}{2} \log q \ll \xi q^\frac{1}{2} \log q .
\]
\end{proof}

\begin{lem}
\label{approximation formula}
Let $\chi_1,\chi_2$ be non-principal Dirichlet characters modulo $q>1$. Let $z_1, z_2$ be complex variables and $\tau >0$ be a large parameter. Then for $\re(z_1)>0,\re(z_2)>0$ and $\re(z_1)+\re(z_2)>1$, we have
\begin{align*}
L(z_1,\chi_1) L(z_2,\chi_2) &= \sum_{mn\leq \tau} \frac{\chi_1(m)\chi_2(n)}{m^{z_1}n^{z_2}} \\
&+O\left(\left(1+\abs{z_1}+\abs{z_2}+\abs{z_1z_2}\right)q(\log q)^2 \tau^{-\min\{\re(z_1),\re(z_2)\}}\right) \\
&+\begin{cases}
    O\left(\left( 1+\abs{z_1-z_2}\right)q^\frac{1}{2}(\log q)\tau^{-x_2} \right) & (\re(z_1-z_2)>1) \\
    O\left(\left( 1+\abs{z_1-z_2}\right)q^\frac{1}{2}(\log q)\tau^{-x_2}\log \tau\right) & (\re(z_1-z_2)=1) \\
    O\left(\left( 1+\abs{z_1-z_2}\right)q^\frac{1}{2}(\log q)\tau^{\frac{1-x_1-x_2}{2}}\right) & (\re(z_1-z_2)<1). \\
\end{cases}
\end{align*}
\end{lem}
\begin{proof}
Let $z_1=x_1+iy_1, z_2=x_2+iy_2$. First we assume that $x_1>1$ and $x_2>1$. Then we divide the double sum as:
\begin{align*}
L(z_1,\chi_1) L(z_2,\chi_2) &= \sum_{mn\leq \tau} \frac{\chi_1(m)\chi_2(n)}{m^{z_1}n^{z_2}} +\sum_{mn> \tau} \frac{\chi_1(m)\chi_2(n)}{m^{z_1}n^{z_2}} \\
&= \sum_{mn\leq \tau} \frac{\chi_1(m)\chi_2(n)}{m^{z_1}n^{z_2}}+\left(\sum_{\substack{mn> \tau \\ m,n\leq \tau}} +\sum_{\substack{m \leq \tau\\ n> \tau}}+\sum_{\substack{m > \tau\\ n\leq \tau}}+\sum_{m,n>\tau} \right) \frac{\chi_1(m)\chi_2(n)}{m^{z_1}n^{z_2}}\\
&= \sum_{mn\leq \tau} \frac{\chi_1(m)\chi_2(n)}{m^{z_1}n^{z_2}}+\left(\Sigma_{1,\tau}+\Sigma_{2,\tau}+\Sigma_{3,\tau}+\Sigma_{4,\tau} \right),
\end{align*}
say, where $\Sigma_{j,\tau}=\Sigma_{j,\tau}(z_1,z_2,\chi_1,\chi_2)$ for $j =1,2,3,4$. By partial summation and the P\'{o}lya--Vinogradov inequality, we find that 
\begin{align*}
\Sigma_{2,\tau},\Sigma_{3,\tau} &\ll (1+|z_1|+|z_2|+|z_1z_2|) q(\log q)^2\tau^{-\min\{x_1,x_2\}}, \\
\Sigma_{4,\tau} &\ll (1+|z_1|+|z_2|+|z_1z_2|) q(\log q)^2\tau^{-x_1-x_2}
\end{align*}
uniformly for $x_1>0, x_2>0$, respectively. 

As for $\Sigma_{1,\tau}$, we again divide $\Sigma_{1,\tau}$ into
\begin{align*}
\Sigma_{1,\tau} &= \left( \sum_{\sqrt{\tau}<m,n \leq \tau}+ \sum_{\substack{m\leq \tau\\ \frac{\tau}{m}<n \leq \tau}}+\sum_{\substack{n\leq \tau\\ \frac{\tau}{n}<m \leq \tau}} \right)\frac{\chi_1(m)\chi_2(n)}{m^{z_1}n^{z_2}} \\
&= \Sigma_{11,\tau}+\Sigma_{12,\tau}+\Sigma_{13,\tau},
\end{align*}
say. By the same argument as in $\Sigma_{4,\tau}$, we find that 
\begin{align*}
\Sigma_{11,\tau} \ll (1+|z_1|+|z_2|+|z_1z_2|) q(\log q)^2\tau^{-\frac{x_1+x_2}{2}}
\end{align*}
uniformly for $x_1>0$ and $x_2>0$.

Next, we investigate $\Sigma_{12,\tau}$. By partial summation, we have
\begin{align}
\label{Sigma12}
\begin{split}
\Sigma_{12,\tau} &= \sum_{m \leq \sqrt{\tau}} \frac{\chi_1(m)}{m^{z_1}} \sum_{\frac{\tau}{m}<n\leq \tau} \frac{\chi_2(n)}{n^{z_2}} \\
&=\sum_{m \leq \sqrt{\tau}} \frac{\chi_1(m)}{m^{z_1}} \left( \tau^{-z_2}\sum_{n\leq \tau} \chi_2(n) -\left(\frac{\tau}{m}\right)^{-z_2} \sum_{n\leq \frac{\tau}{m}} \chi_2(n) +z_2\int_{\frac{\tau}{m}}^\tau u^{-z_2-1}\sum_{n\leq u}\chi_2(n) du \right) \\
&=\tau^{-z_2}\sum_{n\leq \tau} \chi_2(n)\sum_{m \leq \sqrt{\tau}} \frac{\chi_1(m)}{m^{z_1}} -\tau^{-z_2} \sum_{m \leq \sqrt{\tau}} \frac{\chi_1(m)}{m^{z_1-z_2}} \sum_{n\leq \frac{\tau}{m}} \chi_2(n) \\
&+z_2 \sum_{m \leq \sqrt{\tau}} \frac{\chi_1(m)}{m^{z_1}} \int_{\frac{\tau}{m}}^\tau u^{-z_2-1}\sum_{n \leq u}\chi_2(n) du.
\end{split}
\end{align}
Let $\Sigma_{121,\tau}, \Sigma_{122,\tau}, \Sigma_{123,\tau}$ denote the last three terms in (\ref{Sigma12}), respectively. By the same argument as in $\Sigma_{4,\tau}$, we find that \begin{align*}
\Sigma_{121,\tau} &\ll (1+\abs{z_1})q(\log q)^2\tau^{-\frac{x_1}{2}-x_2}
\end{align*}
uniformly for $x_1>0,x_2>0$. By partial summation, we have
\begin{align*}
\Sigma_{122,\tau} &= -\tau^{-z_2} \left( \tau^{-\frac{z_1}{2}+\frac{z_2}{2}} \sum_{\substack{mn \leq \tau \\ m \leq \sqrt{\tau}}}\chi_1(m) \chi_2(n)+(z_1-z_2)\int_1^{\sqrt{\tau}} \frac{\sum_{\substack{mn \leq \tau \\ m \leq v}}\chi_1(m)\chi_2(n)}{v^{z_1-z_2+1}} dv \right) \\
&= - \tau^{-\frac{z_1+z_2}{2}} \sum_{\substack{mn \leq \tau \\ m \leq \sqrt{\tau}}}\chi_1(m) \chi_2(n)-(z_1-z_2)\tau^{z_2}\int_1^{\sqrt{\tau}} \frac{\sum_{\substack{mn \leq \tau \\ m \leq v}}\chi_1(m)\chi_2(n)}{v^{z_1-z_2+1}} dv.
\end{align*}
Then by Lemma \ref{hyperbola character sum} and the P\'{o}lya--Vinogradov inequality, we have
\begin{align*}
\Sigma_{122,\tau} &\ll \tau^{\frac{1-x_1-x_2}{2}}\sqrt{q} \log q +\abs{z_1-z_2}\sqrt{q} \log q \times \begin{cases}
\tau^{-x_2} & (x_1-x_2>1) \\
\tau^{-x_2}\log \tau & (x_1-x_2=1) \\
\tau^{\frac{1-x_1-x_2}{2}} & (x_1-x_2<1) \\
\end{cases}
\end{align*}
uniformly for $x_2>0,x_1+x_2>1$.

Finally, we estimate $\Sigma_{123,\tau}$. By the same argument as above, we have
\begin{align}
\label{Sigma_{123,tau}}
\begin{split}
\Sigma_{123,\tau} &= \tau^{-\frac{z_1}{2}}z_2 \sum_{m \leq \sqrt{\tau}} \chi_1(m) \int_{\frac{\tau}{m}}^\tau \frac{\sum_{n \leq u}\chi_2(n)}{u^{z_2+1}} du \\
&-z_1z_2 \int_1^{\sqrt{\tau}} \frac{1}{v^{z_1+1}} \sum_{m \leq v}\chi_1(m) \int_{\frac{\tau}{m}}^\tau \frac{\sum_{n \leq u} \chi_2(n)}{u^{z_2+1}} du dv.
\end{split}
\end{align}
The first term in (\ref{Sigma_{123,tau}}) can be estimated as
\begin{align*}
\tau^{-\frac{z_1}{2}}z_2 \sum_{m \leq \sqrt{\tau}} \chi_1(m) \int_{\frac{\tau}{m}}^\tau \frac{\sum_{n \leq u}\chi_2(n)}{u^{z_2+1}} du&=z_2 \tau^{-\frac{z_1}{2}} \int_{\sqrt{\tau}}^\tau \frac{\sum_{\frac{\tau}{u}<m\leq \tau}\chi_1(m) \sum_{n\leq u}\chi_2(n)}{u^{z_2+1}} du \\
&\ll\abs{z_2}q(\log q)^2\tau^{-\frac{x_1+x_2}{2}}
\end{align*}
for $x_2>0$. As for the second term in (\ref{Sigma_{123,tau}}), it can be rewritten as
\begin{align*}
&-z_1z_2 \int_1^{\sqrt{\tau}} \frac{1}{v^{z_1+1}} \int_{\frac{\tau}{v}}^\tau   \frac{\sum_{\frac{\tau}{u}<m \leq v}\chi_1(m)\sum_{n \leq u} \chi_2(n)}{u^{z_2+1}} du dv.
\end{align*}
Then the above is estimated as
\begin{align*}
&\ll \abs{z_1z_2}q(\log q)^2 \int_1^{\sqrt{\tau}} \frac{1}{v^{x_1+1}} \int_{\frac{\tau}{v}}^\tau \frac{1}{u^{x_2+1}} du dv \\
&\ll \abs{z_1z_2}q(\log q)^2 \left(\tau^{-\frac{x_1}{2}-x_2}+\tau^{-x_2}\begin{cases}
    \tau^{-x_2} & (x_1>x_2) \\
    \tau^{-x_2}\log \tau & (x_1=x_2) \\
    \tau^{-\frac{x_1+x_2}{2}} & (x_1<x_2) \\
\end{cases} \right) \\
&\ll \abs{z_1z_2}q(\log q)^2 \begin{cases}
\tau^{-x_2}\log \tau & (x_1\geq x_2) \\
\tau^{-\frac{x_1+x_2}{2}} & (x_1<x_2) \\
\end{cases}
\end{align*}
for $x_1,x_2>0$. By the symmetry, the term $\Sigma_{13,\tau}$ can be estimated by the same argument as in $\Sigma_{12,\tau}$.

Summarizing the above, for $x_1,x_2>0$ and $x_1+x_2>1$ we have
\begin{align*}
&\Sigma_{1,\tau}+\Sigma_{2,\tau}+\Sigma_{3,\tau}+\Sigma_{4,\tau} \\
&\ll \left(1+\abs{z_1}+\abs{z_2}+\abs{z_1z_2}\right)q(\log q)^2 \tau^{-\min\{x_1,x_2\}} \\
&\quad + \left( 1+\abs{z_1-z_2}\right)q^\frac{1}{2}\log q \begin{cases}
    \tau^{-x_2} & (x_1-x_2>1) \\
    \tau^{-x_2}\log \tau & (x_1-x_2=1) \\
    \tau^{\frac{1-x_1-x_2}{2}} & (x_1-x_2<1). \\
\end{cases}
\end{align*}
\end{proof}

\begin{lem}
\lab{lem:L and conf hypergeometric}
Let $(s_1,s_2,z) \in \mathbb{C}^3$ with $\sigma_1+\sigma_2>1, \re(z)+\sigma_2>1, \sigma_1+\sigma_2+\re(z)>2$ and $\sigma_2>0, \re(z)<0$. Suppose that $\chi_1, \chi_2$ are Dirichlet characters modulo $q$. If $\chi_2$ is primitive, then we have
\begin{align*}
\widetilde{L}_{2,z} (s_1,s_2;\chi_1,\chi_2) 
&= \eps(\chi_2)(2\pi)^{z+s_2-1}q^{\frac{1}{2}-z-s_2}\Gm(1-z) \sum_{m \geq 1} \sum_{n \geq 1} \frac{\chi_1(m)\overline{\chi}_2(n)}{m^{s_1}n^{1-s_2-z}}\\
&\quad \times \left\{ e^{\frac{\pi i}{2}(z+s_2+\kappa_2-1)} \Psi(s_2,s_2+z;2\pi imn/q) \right. \\
&\qquad +\left. e^{-\frac{\pi i}{2}(z+s_2+\kappa_2-1)} \Psi(s_2,s_2+z;-2\pi imn/q) \right\},
\end{align*}
where $\eps(\tau_j)=\tau(\chi)/(i^{\kappa_j}\sqrt{q})$ and $\kappa_j =0$ if $\chi_j$ is even, and $\kappa_j =1$ if $\chi_j$ is odd.
\end{lem}
\begin{proof}
This is a special case of Lemma 3.4 (2) in \cite{T22+}.
\end{proof}

\begin{lem}
\lab{lem:L and conf hypergeometric-2}
Let $(s_1,s_2) \in \mathbb{C}^2$ with $\sigma_2>0, \sigma_1+\sigma_2>0$. Suppose that $\chi_1, \chi_2$ are primitive Dirichlet characters modulo $q$ and $\kappa_2$ is defined as $\chi_2(-1)=(-1)^{\kappa_2}$. If $\kappa_2=0$, then we have
\begin{align}
\begin{split}
\lab{formula:L-integral-1}
&L_2 (s_1,s_2;\chi_1,\chi_2) \\
&= 2(2\pi)^{s_2-1}q^{\frac{1}{2}-s_2} \eps(\chi_2) \sum_{m=1}^\infty \sum_{n=1}^\infty \frac{\chi_1(m) \overline{\chi_2}(n)}{m^{s_1}n^{1-s_2}}\\
&\times \left(\cos\frac{2\pi mn}{q} \int_{\frac{2\pi mn}{q}}^\infty x^{-s_2} \cos xdx +\sin\frac{2\pi mn}{q} \int_{\frac{2\pi mn}{q}}^\infty x^{-s_2} \sin xdx\right). 
\end{split}
\end{align}
Moreover if $\kappa_2=1$, then we have
\begin{align}
\begin{split}
\lab{formula:L-integral-2}
&L_2 (s_1,s_2;\chi_1,\chi_2) \\
&= 2(2\pi)^{s_2-1}q^{\frac{1}{2}-s_2} \eps(\chi_2) \sum_{m=1}^\infty\sum_{n=1}^\infty \frac{\chi_1(m) \overline{\chi_2}(n)}{m^{s_1}n^{1-s_2}}\\
& \times \left(-\sin\frac{2\pi mn}{q} \int_{\frac{2\pi mn}{q}}^\infty x^{-s_2} \cos xdx +\cos\frac{2\pi mn}{q} \int_{\frac{2\pi mn}{q}}^\infty x^{-s_2} \sin xdx\right).
\end{split}
\end{align}
\end{lem}
\begin{proof}
First we assume that $(s_1,s_2,z) \in \mathbb{C}^3$ with $\sigma_1+\sigma_2>1, \re(z)+\sigma_2>1, \sigma_1+\sigma_2+\re(z)>2$ and $\sigma_2>0, \re(z)<0$. Then by  the well-known formula \cite[6.5 (6)]{ErMaObTr}
\begin{equation*}
\Psi(a,c;x)= x^{1-c} \Psi(a-c+1,2-c;x)
\end{equation*}
and Lemma \ref{lem:L and conf hypergeometric} to have
\begin{align}
\label{double series}
\begin{split}
\widetilde{L}_{2,z} (s_1,s_2;\chi_1,\chi_2)&=\frac{\eps(\chi_2)\Gm(1-z)}{\sqrt{q}} \sum_{m \geq 1} \sum_{n \geq 1} \frac{\chi_1(m)\overline{\chi}_2(n)}{m^{s_1+s_2+z-1}}\\
&\quad \times \left\{ e^{\frac{\pi i}{2}\kappa_2}\Psi(1-z,2-s_2-z;2\pi imn/q) \right. \\
&\qquad +\left. e^{-\frac{\pi i}{2}\kappa_2}\Psi(1-z,2-s_2-z;-2\pi imn/q) \right\}.
\end{split}
\end{align}

The asymptotic expansion of $\Psi(a,c;x)$ (see \cite[6.13.1 (1)]{ErMaObTr}) is
\begin{equation*}
\Psi(a,c;x) = \sum_{k=0}^{N-1} \frac{(-1)^k (a)_k (a-c+1)_k}{k!}x^{-a-k} + \rho_N (a,c;x),
\end{equation*}
where $N$ is an arbitrary positive integer, $(a)_k = \Gm(a+k)/\Gm(a)$ and $\rho_N (a,c;x)$ is the remainder term. By applying the above, we find that the double series on the right-hand side of (\ref{double series}) is equal to
\begin{align}
\label{double series-2}
\begin{split}
&\sum_{m \geq 1} \sum_{n \geq 1} \frac{\chi_1(m)\overline{\chi}_2(n)}{m^{s_1+s_2+z-1}} \left\{ e^{\frac{\pi i}{2}\kappa_2}\Psi(1-z,2-s_2-z;2\pi imn/q) \right. \\
&\qquad +\left. e^{-\frac{\pi i}{2}\kappa_2}\Psi(1-z,2-s_2-z;-2\pi imn/q) \right\}\\
&= \sum_{k=0}^N \frac{(-1)^k (1-z)_k(s_2)_k}{\frac{(2\pi}{q})^{k+1-z} k!} \\
&\qquad \times \left( e^{\frac{\pi i}{2}(z-k-1+\kappa_2)}+e^{-\frac{\pi i}{2}(z-k-1+\kappa_2)} \right) L(s_1+s_2+k,\chi_1)L(k+1-z,\overline{\chi}_2)\\
&\quad+\sum_{m \geq 1} \sum_{n \geq 1} \frac{\chi_1(m)\overline{\chi}_2(n)}{m^{s_1+s_2+z-1}} \left\{ e^{\frac{\pi i}{2}\kappa_2} \rho_N(1-z,2-s_2-z;2\pi imn/q) \right. \\
&\qquad +\left. e^{-\frac{\pi i}{2}\kappa_2}\rho_N(1-z,2-s_2-z;-2\pi imn/q) \right\}.
\end{split}
\end{align}
By applying the estimate (\cite[(6.2)]{Ma98})
\begin{align*}
&\lvert \rho_N (1-z, 2-s_2-z; \pm 2\pi imn/q) \rvert  \\
&\ll \frac{\lvert (s_2)_k \rvert \Gm(-\re(z)+N+1)}{N!\lvert\Gm(1-z)\rvert} e^{\pi\frac{\lvert \im(z) \rvert+\lvert t_2 \rvert}{2}} (2\pi mn/q)^{\re(z)-N-1},
\end{align*}
where $\re(z)<N+1$ and $\si_2\geq -N$, we find that the second term on the right-hand side of (\ref{double series-2}) is estimated by
\begin{align*}
&\ll q^{N-\re(z)}\frac{\lvert\tau(\chi_2) (s_2)_k \rvert \Gm(-\re(z)+N+1)}{N!} e^{\pi\frac{\lvert \im(z) \rvert+\lvert t_2 \rvert}{2}} \zeta(\sigma_1+\sigma_2+N) \zeta(1-\re(z)+N)
\end{align*}
for $\re(z)<N$ and $\sigma_1+\sigma_2\geq 1-N$. Thus the second term in (\ref{double series-2}) is absolutely convergent in the region. 

Now we take $N \geq 2$. Since $\chi_1,\chi_2$ are primitive, the first term in (\ref{double series-2}) is convergent when $z<1$ and $\sigma_1+\sigma_2>0$. Therefore the double series on the right-hand side of (\ref{double series}) is uniformly convergent for $\sigma_1+\sigma_2>0,z<1$. Taking $z=0$, if $\sigma_1+\sigma_2>0$ we obtain 
\begin{align}
\label{L_2-conf hypergeo}
\begin{split}
L_2 (s_1,s_2;\chi_1,\chi_2)&=\widetilde{L}_{2,0}(s_1,s_2;\chi_1,\chi_2) \\
&= \frac{\eps(\chi_2)}{\sqrt{q}} \sum_{m \geq 1} \sum_{n \geq 1} \frac{\chi_1(m)\overline{\chi}_2(n)}{m^{s_1+s_2-1}}\\
&\quad \times \left\{ e^{\frac{\pi i}{2}\kappa_2}\Psi(1,2-s_2;2\pi imn/q)+e^{-\frac{\pi i}{2}\kappa_2}\Psi(1,2-s_2;-2\pi imn/q) \right\}.
\end{split}
\end{align}

Now we invoke the classical relation
\[
z^a e^{-z} \Psi (1,a+1;z)=\Gamma(a,z)
\]
(see \cite[6.5.6(2)]{ErMaObTr} and \cite[6.9.2(21)]{ErMaObTr}), where $\Gm(a,z)$ denotes the incomplete Gamma function of the second kind defined by
\[
\Gm(a,z)=\int_z^\infty e^{-t}t^{a-1} dt,
\]
for $\re(a)>0$. Then we have 
\begin{align}
\label{conf hypergeo-incomplete gamma}
\begin{split}
\Psi(1, 2-s_2; \pm2\pi imn/q) &= \left( \pm\frac{2\pi imn}{q}\right)^{s_2-1} e^{\pm\frac{2\pi imn}{q
}}\Gamma\left(1-s_2,\pm \frac{2\pi imn}{q}\right).
\end{split}
\end{align}
Moreover by \cite[9.10.(1),(2)]{ErMaObTr2}, we have
\begin{align}
\label{incomplete gamma-integral}
\begin{split}
e^{-\frac{\pi i}{2}(1-s_2)} \Gamma\left(1-s_2, \frac{2\pi imn}{q}\right) &= \int_{\frac{2\pi mn}{q}}^\infty t^{-s_2} \cos tdt-i\int_{\frac{2\pi mn}{q}}^\infty t^{-s_2} \sin tdt, \\
e^{\frac{\pi i}{2}(1-s_2)} \Gamma\left(1-s_2,- \frac{2\pi imn}{q}\right) &= \int_{\frac{2\pi mn}{q}}^\infty t^{-s_2} \cos tdt+i\int_{\frac{2\pi mn}{q}}^\infty t^{-s_2} \sin tdt
\end{split}
\end{align}
for $\sigma_2>0$.

Combining the (\ref{L_2-conf hypergeo}), (\ref{conf hypergeo-incomplete gamma}) and (\ref{incomplete gamma-integral}), we obtain the desired result. Finally, it should be remarked that these integral expression can be justified by integration by parts.
\end{proof}

\section{Proof of the Theorem}
Now we prove Theorem \ref{thm:2}. In the following proof, the evaluation of $J_1,J_2,J_3,J_4$ (defined below) is the same as in the proof of Theorem 2 in \cite{KTZ11}, While we need Lemma \ref{approximation formula} to evaluate $J_5$ in the odd character case. In order to prove Theorem \ref{thm:2}, we apply the following estimates due to Hardy and Littlewood. 
\begin{lem}\cite[Lemma 12]{HL}
Suppose that $\sigma$ is fixed, $0<\sigma<1$, and $A_0,A_1$ are fixed constants such that $0<A_0<1$ and $1<A_1$ respectively. Consider the integral
\[
I(\xi,s) = \int_\xi^\infty u^{-s} \begin{cases}
\cos u \\
\sin u 
\end{cases} du.
\]
Then for $s=\sigma+it$, we have
\begin{align}
\lab{HL-1}
I(\xi,s) &= \Gm(1-s) \begin{cases}
\sin \frac{\pi s}{2} \\
\cos \frac{\pi s}{2}
\end{cases} +O\left(\frac{\xi^{1-\sigma}}{|t|} \right) & (\xi<A_0 |t|<|t|), \\
\lab{HL-2}
I(\xi,s) &= \Gm(1-s) \begin{cases}
\sin \frac{\pi s}{2} \\
\cos \frac{\pi s}{2}
\end{cases} +O\left(\frac{\xi^{2-\sigma}}{|t|(|t|-\xi)} \right) & (A_0 |t|<\xi<|t|), \\
\lab{HL-3}
I(\xi,s) &= O\left(\frac{\xi^{1-\sigma}}{\xi-|t|} \right) & (|t|<\xi<A_1 |t|), \\
\lab{HL-4}
I(\xi,s) &= O\left(\xi^{1-\sigma} \right) & (|t|<A_1|t|<\xi), \\
\intertext{and}
\lab{HL-5}
I(\xi,s) &= O\left(\xi^{1-\sigma} |t|^\frac{1}{2}\right)
\end{align}
in any case.
\end{lem}

\begin{proof}[Proof of Theorem \ref{thm:2}]
We may suppose that $t_2 > 0$. Let $\tau$ be a parameter such that $1 \ll \tau \ll (qt_2)^\frac{1}{2}$. We divide the double sum in (\ref{formula:L-integral-1}) and (\ref{formula:L-integral-2}) into the following five parts:
\begin{align*}
&\sum_{mn\leq \frac{A_0 qt_2}{2\pi}} +\sum_{\frac{A_0 qt_2}{2\pi}<mn\leq \frac{qt_2}{2\pi} -\tau} +\sum_{\frac{qt_2}{2\pi}-\tau<mn\leq \frac{qt_2}{2\pi} +\tau} +\sum_{\frac{qt_2}{2\pi}+\tau<mn\leq \frac{A_1 qt_2}{2\pi}} +\sum_{\frac{A_1 qt_2}{2\pi}<mn} \\
&= J_1+J_2+J_3+J_4+J_5,
\end{align*}
say. We first assume that $\kappa_2=0$. By the same argument as in \cite{KTZ11}, we have
\begin{align*}
J_1&=2\eps(\chi_2)(2\pi)^{s_2-1}q^{\frac{1}{2}-s_2} \sum_{mn\leq \frac{A_0 qt_2}{2\pi}} \frac{\chi_1(m) \overline{\chi_2}(n)}{m^{s_1}n^{1-s_2}} \\
&\quad \times \left(\Gm(1-s_2)\left(\cos\frac{2\pi mn}{q}\sin \frac{\pi s_2}{2}+\sin\frac{2\pi mn}{q}\cos \frac{\pi s_2}{2} \right) + O \left(\frac{1}{t_2}\left(\frac{mn}{q}\right)^{1-\sigma_2}\right) \right) \\
&=2\eps(\chi_2)(2\pi)^{s_2-1}q^{\frac{1}{2}-s_2}\Gm(1-s_2) \left( \sin \frac{\pi s_2}{2}\sum_{mn\leq \frac{A_0 qt_2}{2\pi}} \frac{\chi_1(m) \overline{\chi_2}(n)\cos\frac{2\pi mn}{q}}{m^{s_1}n^{1-s_2}} \right. \\
&\quad \left.+\cos \frac{\pi s_2}{2}\sum_{mn\leq \frac{A_0 qt_2}{2\pi}} \frac{\chi_1(m) \overline{\chi_2}(n)\sin\frac{2\pi mn}{q} }{m^{s_1}n^{1-s_2}} \right) + O \left( q^\frac{1}{2}(qt_2)^{\delta+\eps}\right),
\end{align*}
where $\delta = \max \{ 0, 1-\sigma_1-\sigma_2 \}$. As for $J_2$, we obtain
\begin{align*}
J_2&=2\eps(\chi_2)(2\pi)^{s_2-1}q^{\frac{1}{2}-s_2}\Gm(1-s_2) \\
&\times \left( \sin \frac{\pi s_2}{2}\sum_{mn\leq \frac{A_0 qt_2}{2\pi}} \frac{\chi_1(m) \overline{\chi_2}(n)\cos\frac{2\pi mn}{q}}{m^{s_1}n^{1-s_2}} +\cos \frac{\pi s_2}{2}\sum_{mn\leq \frac{A_0 qt_2}{2\pi}} \frac{\chi_1(m) \overline{\chi_2}(n)\sin\frac{2\pi mn}{q} }{m^{s_1}n^{1-s_2}} \right) \\
&\quad + O\left(\frac{1}{t_2q^{\frac{3}{2}}} \sum_{\frac{A_0 qt_2}{2\pi}<\ell\leq \frac{qt_2}{2\pi}-\tau} \frac{|\sigma_{1-s_1-s_2}(\ell)| \ell}{t_2-\frac{2\pi \ell}{q}} \right),
\end{align*}
where $\sigma_\alpha (n) = \sum_{d \mid n} d^\alpha$. The $O$-term above is estimated as
\begin{align*}
\ll q^{-\frac{1}{2}} (qt)_2^{\delta+\varepsilon} \sum_{1 \leq j \ll qt_2}\frac{1}{\tau+j} \ll  q^{-\frac{1}{2}}(qt_2)^{\delta+\varepsilon}.
\end{align*}
By (\ref{HL-3}) and by the same method as in the case of $J_2$, we have 
\[
J_4 = O \left(\  q^{-\frac{1}{2}}(qt_2)^{\delta+\varepsilon} \right).
\]
By the same calculations as in \cite{KTZ11}, we estimate $J_3$ as
\[
J_3 \ll (qt_2)^\frac{1}{2}\sum_{|\ell -\frac{qt_2}{2\pi}| \leq \tau} \frac{|\sigma_{1-s_1-s_2}(\ell)|}{\ell} \ll (qt_2)^{\delta-\frac{1}{2}+\eps} \tau \ll q^{-\frac{1}{2}}(qt_2)^{\delta+\varepsilon}.
\]

Next, we estimate $J_5$. We apply
\begin{align*}
\int_{\frac{2\pi mn}{q}}^\infty x^{-s_2} \cos xdx&=-\left(\frac{2\pi mn}{q}\right)^{-s_2}\sin\frac{2\pi mn}{q} +O\left( t_2 \left(\frac{mn}{q}\right)^{-\sigma_2-1}\right), \\
\int_{\frac{2\pi mn}{q}}^\infty x^{-s_2} \sin xdx&=\left(\frac{2\pi mn}{q}\right)^{-s_2}\cos\frac{2\pi mn}{q} +O\left( t_2 \left(\frac{mn}{q}\right)^{-\sigma_2-1}\right)
\end{align*}
to obtain
\begin{align*}
J_5 &\ll  q^\frac{1}{2} t_2 \sum_{\ell>\frac{A_1 qt_2}{2\pi}} \frac{|\sigma_{1-s_1-s_2}(\ell)|}{\ell^2} \ll q^{-\frac{1}{2}}(qt_2)^{\delta+\eps}.
\end{align*}

Lastly, we note that 
\[
\left| \eps(\chi_2)(2\pi)^{s_2-1}q^{\frac{1}{2}-s_2}\Gm(1-s_2)\left( \sin \frac{\pi s_2}{2}+\cos\frac{\pi s_2}{2}\right)\right| \sum_{\frac{qt_2}{2\pi}-\tau <\ell <\frac{qt_2}{2\pi}} \frac{|\sigma_{1-s_1-s_2}(\ell)|}{\ell^{1-\sigma_2}}\ll (qt_2)^{\delta+\eps}.
\]
Therefore, we get the desired result.

Next, we assume that $\kappa_2=1$. We can estimate $J_1,J_2,J_3,J_4$ by the same argument as above. As for $J_5$, by Lemma \ref{approximation formula} with $z_1=1$ and $z_2=s_1+s_2$, we have
\begin{align}
\label{J_5 chi_2:odd}
\begin{split}
J_5 &=\frac{2\eps(\chi_2)q^{\frac{1}{2}}}{2\pi} \sum_{mn>\frac{A_1qt_2}{2\pi}} \frac{\chi_1(m)\overline{\chi_2}(n)}{m^{s_1+s_2}n} +O \left(q^\frac{1}{2}t_2\sum_{\ell>\frac{A_1 qt_2}{2\pi}} \frac{|\sigma_{1-s_1-s_2}(\ell)|}{\ell^2} \right) \\
&\ll(1+\abs{t_1+t_2}) q^{\frac{3}{2}+\eps} (qt_2)^{-\min \{1,\si_1+\si_2\}} \log (qt_2) \\
&\quad+(1+\abs{t_1+t_2}) q^{1+\eps} \begin{cases}
    (qt_2)^{-1} & (\si_1+\si_2>2) \\
    (qt_2)^{-1}\log(qt_2) & (\si_1+\si_2=2) \\
    (qt_2)^{-\frac{\si_1+\si_2}{2}} & (\si_1+\si_2<2) \\
\end{cases} \\
&\quad +q^{-\frac{1}{2}}(qt_2)^{\delta+\eps} \\
&\ll (1+\abs{t_1+t_2})q^{\frac{3}{2}+\eps}(qt_2)^{-\min \{1,\frac{\si_1+\si_2}{2} \}+\eps}+q^{-\frac{1}{2}}(qt_2)^{\delta+\eps}
\end{split}
\end{align}
for $\si_1+\si_2>0$.
\end{proof}

\begin{ack} 
The author would like to thank Professor Kohji Matsumoto for valuable comments. This work is supported by Nagoya University Interdisciplinary Frontier Fellowship (Grant Number JPMJFS2120).
\end{ack} 

\bibliographystyle{plain}

\end{document}